\documentclass[12pt]{article}

\usepackage{amsmath}
\usepackage{amssymb}
\usepackage{amsthm}
\usepackage{graphicx}
\usepackage{tikz}
\usepackage{caption}
\usepackage{fourier} 
\usepackage{soul} 
\usepackage{hyperref}

\newtheorem{theorem}{Theorem}
\newtheorem{proposition}{Proposition}
\newtheorem{corollary}{Corollary}
\newtheorem{lemma}{Lemma}

\theoremstyle{definition}

\theoremstyle{remark}

\def\11{\mathds{1}}

\def\E{\mathbb{E}}
\def\P{\mathbb{P}}
\def\R{\mathbb{R}}

\def\N{\mathbb{N}}

\def\d{\partial}

\def\cF{{\cal F}}
\def\cG{{\cal G}}

\begin{document}

\title{Non-failable approximation method for conditioned distributions}
\author{William O\c{c}afrain$^{1}$ and Denis Villemonais$^{1,2,3}$}

\footnotetext[1]{\'Ecole des Mines de Nancy, Campus ARTEM, CS 14234, 54042 Nancy Cedex, France}
\footnotetext[2]{IECL, Universit\'e de Lorraine, Site de Nancy, B.P. 70239, F-54506 Vandœuvre-lès-Nancy Cedex, France}
\footnotetext[3]{Inria, TOSCA team, Villers-l\`es-Nancy, F-54600, France.\\
  E-mail: william.ocafrain9@etu.univ-lorraine.fr, denis.villemonais@univ-lorraine.fr}

\maketitle

\begin{abstract}
We consider a general method for the approximation of the distribution of a process conditioned to not hit a given set. Existing methods are based on particle system that are failable, in the sense that, in many situations, they are not well defined after a given random time. We present a method based on a new particle system which is always well define. Moreover, we provide sufficient conditions ensuring that the particle method converges uniformly in time. We also show that this method provides an approximation method for the quasi-stationary distribution of Markov processes. Our results are illustrated by their application to a neutron transport model.
\end{abstract}

\noindent\textit{Keywords:}{ Particle system; process with absorption; Approximation method for degenerate processes}

\medskip\noindent\textit{2010 Mathematics Subject Classification.} Primary: {37A25; 60B10; 60F99}. Secondary: {60J80}

\section{Introduction}

This article is concerned with the approximation of the distribution of Markov processes conditioned to not hit a given absorbing state. Let $X$ be a discrete time Markov process evolving in a state space $E\cup\{\d\}$, where $\d\notin E$ is an absorbing state, which means that
\begin{align*}
X_n=\d,\ \forall n\geq \tau_\d,
\end{align*}
where $\tau_\d=\min\{n\geq 0,\ X_n=\d\}$. Our first aim is to provide an approximation method based on an interacting particle system for the conditional distribution
\begin{align}
\label{eq:intro:cond-dist}
\P_\mu(X_n\in \cdot\mid n<\tau_\d)=\frac{\P_\mu(X_n\in\cdot \text{ and } n<\tau_\d)}{\P_\mu(n<\tau_\d)},
\end{align}
where $\P_\mu$ denotes the law of $X$ with initial distribution $\mu$ on $E$.
Our only assumption to achieve our aim will be that survival during a given finite time is possible from any state $x\in E$, which means that
\begin{align}
\label{eq:intro:assu}
\P_x(\tau_\d>1)=\P_x(X_1\in E)>0,\ \forall x\in E.
\end{align}
Our second aim is to provide a general condition ensuring that the approximation method is uniform in time. The main assumption will be that there exist positive constants $C$ and $\gamma$ such that, for any initial distributions $\mu_1$ and $\mu_2$,
\begin{align*}
\left\|\P_{\mu_1}(X_n\in\cdot\mid n<\tau_\d)-\P_{\mu_2}(X_n\in\cdot\mid n<\tau_\d)\right\|_{TV}\leq Ce^{-\gamma n},\ \forall n\geq 0.
\end{align*}
This property has been extensively studied in~\cite{CV2016PTRF}. In particular, it is known to imply the existence of a unique quasi-stationary distribution for the process $X$ on $E$. Another main result of our paper is that, under mild assumptions, the approximation method can be used to estimate this quasi-stationary distribution.

The na\"ive Monte-Carlo approach to approximate such distributions would be to consider $N\gg 1$ independent interacting particles $X^1,\ldots,X^N$ evolving following the law of $X$ under $\P_\mu$ and to use the following asymptotic relation
\begin{align*}
\P_\mu(X_n\in\cdot \text{ and } n<\tau_\d)\simeq_{N\rightarrow\infty} \frac{1}{N}\sum_{i=1}^N \delta_{X^i_n}(\cdot\cap E)
\end{align*}
and then
\begin{align*}
\P_\mu(X_n\in\cdot \mid n<\tau_\d) \simeq_{N\rightarrow\infty} \frac{\sum_{i=1}^N \delta_{X^i_n}(\cdot\cap E)}{\sum_{i=1}^N \delta_{X^i_n}(E)}.
\end{align*}
However, the number of particles remaining in $E$ typically decreases exponentially fast, so that, at any time $n\geq 0$, the actual number of particles that are used to approximate $\P_\mu(X_n\in\cdot \mid n<\tau_\d)$ is of order $e^{-\lambda n}N$ for some $\lambda>0$. 
As a consequence, the variance of the right hand term typically grows exponentially fast and then the precision of the Monte-Carlo method  worsens dramatically over time. In fact, for a finite number of particles $N$, the number of particles $X^i_n$ belonging to $E$ eventually vanishes in finite time with probability one. Thus the right hand term in the above equation eventually becomes undefined. Since we're typically interested in the long time behavior of~\eqref{eq:intro:cond-dist} or in methods that need to evolve without interruption for a long time, the na\"ive Monte Carlo method is definitely not well suited to fulfill our objective.

In order to overcome this difficulty, modified Monte-Carlo methods have been introduced in the recent past years by~Del Moral for discrete time Markov processes (see for instance~\cite{DelMoral1998,DelMoral2003} or the well documented web page~\cite{DelMoralWP}, with many applications of such modified Monte-Carlo method).   The main idea is to consider independent particles $X_1,\ldots,X_N$ evolving in $E$ following the law of $X$, but such that, at each time $n\in\N$, any absorbed particle is re-introduced to the position of one other particle, chosen uniformly among those remaining in $E$; then the particles evolve independently from each others and so on. While this method is powerful, one drawback is that, at some random time $T$, all the particles will eventually be absorbed simultaneously. At this time, the interacting particle system is stopped and there is no natural way to reintroduce all the particles at time $T+1$. When the number of particles is large and the probability of absorption is uniformly bounded away from zero, the time $T$ is typically very large and this explain the great success of this method. However, many situations does not enter the scope of these assumptions, such as diffusion processes picked at discrete times or the neutron transport approximation (see Section~\ref{sec:example_neutron}). Our method is non-failable in these situations. Moreover the uniform convergence theorem provided in Section~\ref{sec:main2} also holds in these cases, under suitable assumptions.

When the underlying process is a continuous time process, one alternative to the methods of~\cite{DelMoral1998} has been introduced recently. The idea is to consider a continuous time $N$-particles system, where the particles evolve independently until one (and only one) of them is absorbed. At this time, the unique absorbed particle is re-introduced to the position of one other particle, chosen uniformly among those remaining in $E$. This continuous time system, introduced  by~Burdzy, Holyst, Ingermann and March (see for instance~\cite{Burdzy1996}), can be used to approximate the distribution of diffusion processes conditioned not to hit a boundary. Unfortunately, it  yields two new difficulties. The first one is that it only works if the number of jumps does not explode in finite time almost surely (which is not always the case even in non-trivial situations, see for instance~\cite{BieniekBurdzyPal2012}). The second one is that, when it is implemented numerically, one has to compute the exact absorption time of each particles, which can be cumbersome for diffusion processes and complicated boundaries. Note that, when this difficulties are overcome, the empirical distribution of the process is known to converge to the conditional distribution (see for instance the general result~\cite{Villemonais2014} and the particular cases handled in~\cite{Burdzy2000,Grigorescu2004,
Rousset2006,Ferrari2007,Villemonais2011,
Villemonais2015,Asselah2016}). 

Finally, it appears that both methods are not  applicable in the generality we aim to achieve in the present paper and, in some cases, both method will fail (as in the case of the neutron transport example of Section~\ref{sec:example_neutron}). Let us now describe the original algorithm studied in the present paper.

\medskip

Fix $N\geq 2$. The particle system that we introduce is a discrete time Markov process $(X^1_n,\ldots,X^N_n)_{n\in\N}$ evolving in $E^N$. We describe its dynamic between two successive times $n$ and $n+1$, knowing $(X^1_n,\ldots,X^N_n)\in E^N$, by considering the following random algorithm which act on any $N$-uplet of the form $$y=((x_1,b_1),\ldots,(x_N,b_N))\in (E\times\{0,1\})^N.$$

\medskip\noindent \textbf{Algorithm 1.} Initiate $y$ by setting $x_i=X^i_n$ and $b_i=0$ for all $i\in\{1,\ldots,N\}$ and repeat the following steps until $b_i=1$ for all $i\in\{1,\ldots,N\}$.
\begin{enumerate}
\item Choose randomly an index $i_0$ uniformly among $\{i\in\{1,\ldots,N\},\,b_i=0\}$
\item Choose randomly a position $Z\in E$ according to $\P_{x_{i_0}}(X_1\in\cdot)$. Then
\begin{itemize}
\item If $Z=\d$, chose an index $j_0$ among $\{1,\ldots,N\}\setminus \{i_0\}$ and replace $(x_{i_0},b_{i_0})$ by $(x_{j_0},b_{j_0})$ in $y$.
\item If $Z\neq\d$, replace $(x_{i_0},b_{i_0})$ by $(Z,1)$ in $y$.
\end{itemize}
\end{enumerate}
After a (random) finite number of iterations, the $N$-uplet $y$ will satisfy $b_i=1$ for all $i\in\{1,\ldots,N\}$. When this is achieved, we set $(X^1_{n+1},\ldots,X^N_{n+1})=(x_1,\ldots,x_N)$.

\bigskip
Our first main result, stated in Section~\ref{sec:main1}, is that, for all $n\geq 0$, the empirical distribution of the particle system evolving following the above dynamic actually converges to the conditional distribution of the original process $X$ at time $n$. We prove this result by building a continuous time Markov process $(Y_t)_{t\in[0,+\infty[}$ such that $Y_n$ is distributed as $X_n$ for all entire time $n\in\N$, and such that the general convergence result of~\cite{Villemonais2014} applies. 

Our second main result, stated in Section~\ref{sec:main2}, shows that, if the conditional distribution of the process $X$ is exponentially mixing (in the sense of~\cite{CV2016PTRF} or~\cite{CV2016a} for the time-inhomogeneous setting) and under a non-degeneracy condition that is usually satisfied, then the approximation method converges uniformly in time.

In Section~\ref{sec:example_neutron}, we illustrate our method by proving that it applies to neutron transport process absorbed at the boundary of an open set $D$.

\section{Convergence of fixed time marginals}
\label{sec:main1}

In this section, we consider the particle system defined by Algorithm~1. We state and prove our main result in a general setting.

\begin{theorem}
\label{thm:intro-main}
Assume that $\mu_0^N:=\frac{1}{N}\sum_{i=0}^N \delta_{X^i_0}$ converges in law to a probability measure $\mu_0$ on $E$. Then, for any $n \in \mathbb{N}$ and any bounded continuous function $f:E\rightarrow \R$,
\begin{equation*}
\frac{1}{N} \sum_{i=1}^N f(X^i_n) \xrightarrow[N\rightarrow\infty]{law} \mathbb{E}_{\mu_0}(f(X_n) | n<\tau_\partial).
\end{equation*}
Moreover,
\begin{align}
\label{eq:conv-rate}
\E\left|\frac{1}{N} \sum_{i=1}^N f(X^i_n) - \mathbb{E}_{\mu_0^N}(f(X_n) | n<\tau_\partial)\right|\leq \frac{2(1+\sqrt{2})\|f\|_\infty}{\sqrt{N}}\,\E\left(\frac{1}{\P_{\mu_0^N}\left(n<\tau_\d\right)}\right)
\end{align}
\end{theorem}
 We emphasize that our result applies to any process $X$ satisfying~\eqref{eq:intro:assu}, overcoming the limitations of all previously cited particle approximation methods, as illustrated by the application to a neutron transport process in Section~\ref{sec:example_neutron}.

\begin{proof}[Proof of Theorem~\ref{thm:intro-main}]

The proof is divided in two steps. First, we provide an implementation of Algorithm~1 as the discrete time included chain of a continuous time Fleming-Viot type particle system. In particular, this step provides a mathematically tractable implementation of  Algorithm~1. In a second step, we use existing results on Fleming-Viot type particle systems to deduce that the empirical distribution of the particle system converges to the conditional distribution~\eqref{eq:intro:cond-dist}.

\bigskip\noindent\textit{Step 1 : Algorithm~1 as a Fleming-Viot type process}\\
Let us introduce the continuous time process $(Y_t)_{t\in[0,+\infty[}$ defined, for any $t \in[0,+\infty[$, by
\begin{equation*}
 Y_t = X_{[t]}1_{t<u_{[t]}} + X_{[t]+1}1_{t \geq u_{[t]}}
\end{equation*}
where $[\cdot]$ denotes the integer part and $(u_n)_{n \in \mathbb{N}}$ is a family of independent random variables such that, for all $n \in \mathbb{N}$, $u_n$ follows a uniform law on $[n,n+1]$. With this definition, $(Y_t)_{t\geq 0}$ is a non-Markovian continuous time process such that $Y_n$ and $X_n$ have the same law for all $n\in\N$.

Now, we define the continuous time process $(B_t)_{t \geq 0}$ by
\begin{equation*}
B_t = \underset{n \in \mathbb{N}}{\sum} 1_{t \in [u_n,n+1[},\ \quad\forall t\geq 0.
\end{equation*}
By construction, the continuous-time process $(Z_t)_{t\in[0,+\infty[}$ defined by 
\begin{align*}
Z_t=(t,Y_t,B_t),\ \forall t\in[0,+\infty[,
\end{align*}
 is a strong Markov process evolving in $F=\mathbb{R}_+ \times (E\cup\{\d\}) \times \{0,1\}$, with   absorbing set $\d_F:=\mathbb{R}_+ \times \{\partial\} \times \{0,1\}$ (see Figure~\ref{fig:graph} for an illustration when $E=\N$ and $\d=0$).

 \begin{figure}
\includegraphics[scale=0.8]{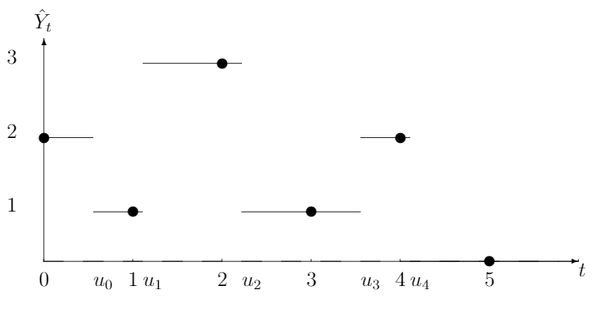}
\caption{The points represent the values taken by the Markov chain $X$ and the  thin lines are the trajectory of $Y$, which jumps from $X_n$ to $X_{n+1}$ at time $u_n$. At any time $t\in[n,n+1[$,  $B_t$ is equal to $1$ if the process has already jumped during any interval $[n, n+1]$ and is equal to $0$ otherwise.}
\label{fig:graph}
\end{figure}

Let us now define a Fleming-Viot type system whose particles evolve as independent copies of $Z$ between their absorption times. More precisely, fix $N\geq 2$ and consider the following continuous time Fleming-Viot type particle system, denoted by $(Z^{i})_{i \in \{1, \ldots, N\}}$, starting from $(z_1,\ldots,z_N)\in F^N$ and evolving as follows.
\begin{itemize}
\item The $N$ particles evolve as $N$ independent copies of $Z$ until one of them reaches $\d_F$. Note that it is clear from the definition of $Z$ that only one particle jumps at this time.
\item Then the unique killed particle is taken from the absorbing point $\d_F$ and is instantaneously placed at the position of an other particle chosen uniformly between the $N-1$ remaining ones; in this situation we say that the particle undergoes a \textit{rebirth}.
\item Then the particles evolve as independent copies of $Z$ until one of them reaches $\d_F$ and so on.
\end{itemize}

For all $t\in[0,+\infty[$, we denote by $A_t^{i,N}$ the number of rebirths of the $i^{th}$ particle occurring before time $t$ and by $A_t^{N}$ the total number of rebirths before the time $t$. Clearly,
\begin{equation*}
 A_t^{N} = \underset{i=1}{\overset{N}{\sum}} A_t^{i,N}\ \text{ almost surely.}
\end{equation*} 
Also, for all $t\in[0,+\infty[$, we set $Z^i_t=(t,Y^i_t,B^i_t)$, where $Y^i_t$ and $B^i_t$ are the marginal component of $Z^i$ in $E$ and $\{0,1\}$ respectively.

One can easily check that this Fleming-Viot system (considered at discrete times) is a particular implementation of the informal description of  Algorithm~1 in the introduction.
 Indeed, at any time $n\geq 0$, the Fleming-Viot system is defined so that
\begin{align*}
(Y^i_n,B^i_n)=(Y^i_n,0).
\end{align*}
Then, at each time $t\in[n,n+1[$, the index of the next moving particle $i\in\{1,\ldots,n\}$ belongs to the set of particles $j\in\{1,\ldots,N\}$ such that $B^j_n=0$. Moreover, conditionally to $B^j_t=0$,  the jumping times of these particles are independent and identically distributed (uniformly on $[t,n+1[$). As a consequence $i$ is chosen uniformly among these indexes (this is the first step of Algorithm~1). Then, at the jumping time $\tau$, the position $Y^i_\tau$ of the particle $i$ is chosen according to $\P_{Y^i_t}(X_1\in\cdot)$ and $B^i_\tau$ is set to $1$. If the position at time $\tau$ is $\d$, then the particle $i$ undergoes a rebirth and hence $(Y^i_\tau,B^i_\tau)$ is replaced by $(Y^k_\tau,B^k_\tau)=(Y^k_t,B^k_t)$, where $k$ is chosen uniformly among $\{1,\ldots,N\}\setminus \{i\}$. Hence the second step of Algorithm~1 is completed. Finally, the procedure is repeated until all the marginals $B^i$ are equal to $1$, as in Algorithm~1.

In particular, for any $n\geq 0$, the random variable $(X^1_n,\ldots,X^N_n)$ obtained from Algorithm~1 and the variable $(Y^1_n,\ldots,Y^N_n)$ obtained from the Fleming-Viot type algorithm have the same law.

\bigskip\noindent\textit{Step 2 : Convergence of the empirical system.}

In this step, we consider a sequence of initial positions $(Y^1_0,\ldots,Y^N_N)_{N\geq 2}$ such that $\frac{1}{N}\sum_{i=0}^N\delta_{Y^i_0}$ converges in law to a probability measure $\mu_0$ on $E$. Our aim is to prove that,  for any $n \in \mathbb{N}$ and any bounded continuous function $f:E\rightarrow \R$,
\begin{equation*}
\frac{1}{N} \sum_{i=1}^N f(Y^i_n) \xrightarrow[N\rightarrow\infty]{law} \mathbb{E}_{\mu_0}(f(Y_n) | n<\tau_\partial).
\end{equation*}
Note that, since $X_n$ and $Y_n$ share the same law, this immediately implies the first part of Theorem~\ref{thm:intro-main}.

Since $(Z^1,\ldots,Z^N)$ is a Fleming-Viot type process without simultaneous killings,  \cite[Theorem~2.2]{Villemonais2014} implies that it is sufficient to prove that, for all $N\geq 2$ and almost surely,
\begin{align}
\label{eq:assumption}
A^N_t<\infty,\ \forall t\geq 0,
\end{align}
where we recall that $A^N_t$ is number of rebirths undergone by the Fleming-Viot type system with $N$ particles before time $t$.

First, let us remark that
\begin{equation*}
\mathbb{P}(A_t^N = \infty) \leqslant \mathbb{P}(A_{[t]+1}^N = \infty).
\end{equation*}
 Moreover,
\begin{align*}
\mathbb{P}(A_{[t]+1}^N = \infty) &= \mathbb{P}\left(\bigcup_{i=0}^{[t]}\{A^N_{i+1}-A_i^N = \infty\}\right) 
 \leqslant \sum_{i=0}^{[t]}\mathbb{P}\left(A_{i+1}^N-A_i^N = \infty\right).
\end{align*}
Using the weak Markov property at time $i$, it is sufficient to prove that 
\begin{equation}
\label{eq:equalzero}
\mathbb{P}(A_{1}^N = \infty)=0
\end{equation}
for any initial distribution of the Fleming-Viot type process $(Y^1,\ldots,Y^N)$ in order to conclude that 
\begin{align*}
\P\left(A_{i+1}^N-A_i^N = \infty\right)=0,\ \forall i\geq 0.
\end{align*}
and hence that~\eqref{eq:assumption} holds true.

But $A_{1}^N = \infty$ if and only if there exists at least one particle for which there is an infinity of rebirths, hence
\begin{equation}
\mathbb{P}(A^N_t = \infty) \leqslant \underset{i=1}{\overset{N}{\sum}} \mathbb{P}\left(A_{1}^{i,N} = \infty\right)
\label{majoration}
\end{equation}
Now, when a particle undergoes a rebirth, it jumps on the position of one of the $N-1$ remaining particles. As a consequence, at any time $t\in[0,1[$, the position $Y^i_t$ of the particles $i\in\{1,\ldots,N\}$ such that $B^i_t=0$ are included in the set $\{Y^1_0,\ldots,Y^N_0\}$. In particular, the probability that such a particle undergoes a rebirth during its next move is bounded above by
\begin{align*}
c:=\max_{i\in\{1,\ldots,N\}} \P_{Y^i_0}\left(X_1=\d\right)<1.
\end{align*}
Hence, a classical renewal argument shows that the probability that a particle undergoes $n$ rebirths is bounded above by $c^n$. This implies that the probability that a particle undergoes an infinity of rebirths is zero. This, together with~\eqref{majoration} implies~\eqref{eq:equalzero}, which concludes the proof of the first part of Theorem~\ref{thm:intro-main}.

\medskip In order to conclude the proof, let us simply remark that, for a deterministic value of $\mu_0^N$, the inequality of Theorem~\ref{thm:intro-main} is directly provided by~\cite[Theorem~2.2]{Villemonais2014}. Now, if $\mu_0^N$ is a random measure, the inequality is obtained by integrating the deterministic case inequality with respect to the law of $\mu_0^N$. This concludes the proof of Theorem~\ref{thm:intro-main}.

\end{proof}

\section{Uniform convergence for uniformly mixing conditioned semi-groups}
\label{sec:main2}

In a recent paper~\cite{CV2016PTRF}, necessary and sufficient conditions on an absorbed Markov process $X$ were obtained to ensure that a process satisfies
\begin{align}
\label{eq:expo-cv}
\left\|\P_{\mu_1}(X_n\in\cdot\mid n<\tau_\d)-\P_{\mu_2}(X_n\in\cdot\mid n<\tau_\d)\right\|_{TV}\leq Ce^{-\gamma n},
\end{align}
where $\gamma$ and $C$ are positive constants. In particular this implies the existence of a unique quasi-stationary distribution $\nu_{QSD}$ for $X$, that is a unique probability measure on $E$ such that $\P_{\nu_{QSD}}(X_n\in\cdot\mid n<\tau_\d)=\nu_{QSD}$, for all $n\geq 0$. General and classical results on quasi-stationary distributions (see for instance~\cite{meleard-villemonais-12,vanDoorn2013,ColletMartinezSanMartin2013}) implies that there exists $\lambda_0>0$ such that
\begin{align}
\label{eq:lambda0}
\P_{\nu_{QSD}}(n<\tau_\d)=e^{-\lambda_0 n},\ \forall n\geq 0.
\end{align}

The exponential convergence property~\ref{eq:expo-cv} holds for a large class of processes, including birth and death processes with catastrophe, branching Brownian particles, neutron transport approximations processes (see~\cite{CV2016PTRF}), one dimensional diffusion with or without killing (see~\cite{CV2016b,CV2016c}), multi-dimen\-sional birth and death processes (see~\cite{CV2015PNMmd}) and multi-dimen\-sional diffusion processes (see~\cite{CCV2016}). Also, similar properties can be proved for time-inhomo\-geneous processes, as stressed in the recent paper~\cite{CV2016a}, with applications to time-inhomo\-geneous diffusion processes and birth and death processes in a quen\-ched random environment. 

In this section, we state and prove our second main result, which states that, if~\eqref{eq:expo-cv} holds and if, for any $n_0\geq 1$, there exists $\varepsilon_0>0$ such that
\begin{align}
\label{eq:additional_assumption}
\inf_{x\in E} \P_x(\P_{X_1}(n_0<\tau_\d)\geq \varepsilon_0\mid 1<\tau_\d)=c_0>0,
\end{align}
then the convergence of the empirical distribution of the particle system described in Algorithm~1 converges uniformly in time to the conditional distribution of the process $X$.

We emphasize that the additional assumption~\eqref{eq:additional_assumption} is true for many processes satisfying~\eqref{eq:expo-cv}, for instance in the case of one-dimensional diffusion processes, multidimensional diffusion processes or piecewise deterministic Markov processes (this the detailed examples of Section~\ref{sec:example_neutron}). 
 As far as we know, none of this processes were covered in this generality by previous methods. In particular, this is the first method that allows the approximation of the conditional distribution of the neutron transport approximation process (see  Section~\ref{sec:example_neutron}), since in this case it easy to check that, with probability one, all the particles will eventually hit the boundary at the same time when using previous algorithms.
The methods also allows to handle the case of the diffusion process on $E=(0,2]$ killed at $0$, reflected at $2$ and solution to the following stochastic differential equation
\begin{align*}
dX_t=dW_t+\frac{1}{\beta X_t^{\beta-1}},\ X_0\in(0,2],\ \beta>2.
\end{align*}
 In this case, the continuous time Fleming-Viot approximation method introduced in~\cite{Burdzy1996} explodes in finite time almost surely, as proved in~\cite{BieniekBurdzyPal2012}. As a matter of fact, it is not known if the Fleming-Viot type particle system is well defined as soon as the diffusion coefficient is degenerated or not regular toward the boundary $0$. On the contrary, our assumption holds true for fairly general one dimensional diffusion processes, thanks to the study provided in~\cite{CV2016b}. Hence our approximation method is valid and, using the next results, converges uniformly in time for both neutron transport processes and degenerate diffusion processes.
 
For any $n\in\N$, we define the empirical distribution of the process at time $n$ as $\mu_n^N = \frac{1}{N} \underset{i=1}{\overset{N}{\sum}} \delta_{X_n^i}$, and for any bounded measurable function $f$ on $E$, we set
\begin{align*}
\mu_n^N(f) = \frac{1}{N} \sum_{i=1}^{N} f(X_n^i).
\end{align*}
\smallskip

\begin{theorem}
\label{thm:main-two}
Assume that~\eqref{eq:expo-cv} and~\eqref{eq:additional_assumption} holds true. Then there exist two constants $C>0$ and $\alpha<0$, such that, for all $\delta>0$ and all measurable function $f:E\rightarrow\R$ bounded by $1$,
\begin{align*}
\E\left(\left|\mu^N_n(f)-\E_{\mu_0^N}(f(X_n)\mid n<\tau_\d)\right|\right)&\leq \frac{C N^\alpha}{\delta}+C\,\P\left(\P_{\mu^N_{0}}(n_0<\tau_\d)\leq \varepsilon_0\delta\right),\ \forall n\geq 0,
\end{align*}
with
\begin{align*}
\alpha=\frac{-\gamma}{2\left(\lambda_0+\gamma\right)}<0.
\end{align*}
\end{theorem}

In the case where the initial position of the particle system are drawn as independent  random variables distributed following the same law $\mu$, then, choosing $\delta>0$ small enough so that $\P_\mu(n_0<\tau_\d)\geq 2\varepsilon_0\delta$, basic concentration inequalities and the equality $\E(\P_{X^i_0}(n_0<\tau_\d))=\P_\mu(n_0<\tau_\d)$  imply that
\begin{align*}
\P\left(\P_{\mu^N_{0}}(n_0<\tau_\d)\leq \varepsilon_0\delta\right)&\leq \P\left(\P_{\mu}(n_0<\tau_\d)-\frac{1}{N}\sum_{i=1}^N\P_{X^i_0}(n_0<\tau_\d)\geq \varepsilon_0\delta\right)\\
&\leq  e^{-N\beta},
\end{align*}
for some $\beta>0$. This implies the following corollary.

\begin{corollary}
Assume that $(X^1_0,\ldots,X^N_0)$ are independent and identically distributed following a given law $\mu$ on $E$. If~\eqref{eq:expo-cv} and~\eqref{eq:additional_assumption} hold true, then there exist two constants $C_\mu>0$ and $\alpha<0$ such that,  and all measurable function $f:E\rightarrow\R$ bounded by $1$,
\begin{align*}
\E\left(\left|\mu^N_n(f)-\E_{\mu_0^N}(f(X_n)\mid n<\tau_\d)\right|\right)&\leq C_\mu N^\alpha,\ \forall n\geq 0,
\end{align*}
where $\alpha$ is the same constant as in Theorem~\ref{thm:main-two}.
\end{corollary}

We emphasize that the above results and their proofs can be adapted to the time-inhomogeneous setting of~\cite{CV2016a}, with appropriate modifications of Assumption~\eqref{eq:additional_assumption}. 

The following result is specific to the time-homogeneous setting and is proved at the end of this section.
\begin{theorem}
\label{thm:invariant-quasi-stationary-distribution} Under the assumptions of Theorem~\ref{thm:main-two}, the particle system $(X^1,\ldots,X^N)$ is exponentially ergodic, which means that it admits a stationary distribution $M^N$ (which is a probability measure on $E^N$) and that there exists positive constants $C_N$ and $\gamma_N$ such that
\begin{align*}
\left\|M^N-\text{Law}(X^1_n,\ldots,X^N_n)\right\|_{TV}\leq C_N e^{-\gamma_N t}.
\end{align*}
Moreover, there exists a positive constant $C>0$ such that,  for all measurable function $f:E\rightarrow\R$ bounded by $1$,
\begin{align}
\label{eq:dist_MN_QSD}
\E_{M^N}\left(\left|\mu^N_0(f)-\nu_{QSD}(f)\right|\right)\leq CN^\alpha,
\end{align}
where $\alpha$ is the same as in Theorem~\ref{thm:main-two} and $(X^1_0,\ldots,X^N_0)$ is distributed following $M^N$.
\end{theorem}

\begin{proof}[Proof of Theorem~\ref{thm:main-two}]

Using the exponential convergence assumption~\eqref{eq:expo-cv}, we deduce that, for any function $f:E\rightarrow\R$ such that $\|f\|_\infty=1$ and all $n\geq n_1\geq 0$,
\begin{align}
\E\left(\left|\mu^N_n(f)-\E_{\mu_0^N}(f(X_n)\mid n<\tau_\d)\right|\right)&\leq \E\left(\left|\mu^N_n(f)-\E_{\mu_{n-n_1}^N}(f(X_{n})\mid n<\tau_\d)\right|\right)\nonumber\\
&\phantom{\leq \E}+\E\left(\left|\E_{\mu_{n-n_1}^N}(f(X_{n})\mid n<\tau_\d)-\E_{\mu_0^N}(f(X_n)\mid n<\tau_\d)\right|\right)\nonumber\\
&\leq \E\left(\left|\mu^N_n(f)-\E_{\mu_{n-n_1}^N}(f(X_n)\mid n<\tau_\d)\right|\right)+Ce^{-\gamma n_1}.
\label{eq:firststep}
\end{align}
Denoting by $(\cF_n)_{n\in\N}$ the natural filtration of the particle system $(X^1_n,\ldots,X^N_n)_{n\in\N}$, we deduce from Theorem~\ref{thm:intro-main} that, almost surely,
\begin{align*}
\E\left(\left|\mu^N_n(f)-\E_{\mu_{n-n_1}^N}(f(X_n)\mid n<\tau_\d)\right|\,\mid \cF_{n-n_1}\right)&\leq 2\wedge\frac{2(1+\sqrt{2})}{\P_{\mu^N_{n-n_1}}(n_1<\tau_\d)\sqrt{N}}.
\end{align*}
Hence, for all $\varepsilon>0$,
\begin{multline*}
\E\left(\left|\mu^N_n(f)-\E_{\mu_{n-n_1}^N}(f(X_n)\mid n<\tau_\d)\right|\right)\leq 2\P\left(\P_{\mu^N_{n-n_1}}(n_1<\tau_\d)\leq \varepsilon\right)\\
+\frac{2(1+\sqrt{2})}{\varepsilon\sqrt{N}}\P\left(\P_{\mu^N_{n-n_1}}(n_1<\tau_\d)> \varepsilon\right).
\end{multline*}
But~\cite[Theorem~2.1]{CV2016PTRF} entails the existence of a measure $\nu$ on $E$ and positive constants $c_1,n_0>0$ and $c_2 > 0$ such that for any $n \in \mathbb{N}$ and $x \in E$,
\begin{align*}
\left\{\begin{array}{l}
\P_x(X_{n_0}\in\cdot\mid n_0<\tau_\d)\geq c_1\nu(\cdot) \\
\P_\nu(n < \tau_\partial) \geq c_2 \P_x( n < \tau_\partial)
\end{array}
\right.
\end{align*}
Note that, from now on, $n_0$ is a fixed constant. This entails 
\begin{align*}
\P_\mu(n_1<\tau_\d)&\geq \P_\mu(n_0<\tau_\d) c_1\P_{\nu}(n_1-n_0<\tau_\d)\\
&\geq  \P_\mu(n_0<\tau_\d) c_1 c_2 \max_{x\in E}\P_x(n_1-n_0<\tau_\d)\\
&\geq c_1c_2\P_\mu(n_0<\tau_\d)e^{-\lambda_0 (n_1-n_0)},
\end{align*}
where $\lambda_0>0$ is the constant of~\eqref{eq:lambda0}.
We deduce that
\begin{multline}
\E\left(\left|\mu^N_n(f)-\E_{\mu_{n-n_1}^N}(f(X_n)\mid n<\tau_\d)\right|\right)\leq 2\P\left(\P_{\mu^N_{n-n_1}}(n_0<\tau_\d)\leq c_3\varepsilon e^{(n_1-n_0)\lambda_0}\right)\\
+\frac{2(1+\sqrt{2})}{\varepsilon\sqrt{N}}.\label{eq:last_eq}
\end{multline}
with $c_3 = \frac{1}{c_1c_2}$.
Our aim is now to control $\P\left(\P_{\mu^N_{n}}(n_0<\tau_\d)\leq \varepsilon\right)$, uniformly in $n\geq 0$ and for all $\varepsilon\in(0,\varepsilon_0/2)$. In order to do so, we make use of the following lemma, proved at the end of this subsection.

\begin{lemma}
\label{lem:controle_pos_part}
There exists $p_0>0$ and $\delta>0$ such that, for any value of $\mu_0^N$,
\begin{align*}
\P\left(\mu_1^N(\P_\cdot(n_0<\tau_\d))\leq \delta\varepsilon_0\right)\leq 1-p_0.
\end{align*}
Moreover, if $\P_{\mu^N_0}(1<\tau_\d)> \varepsilon$ for some $\varepsilon\in(0,c_0\varepsilon_0)$, then 
\begin{align*}
\P\left(\mu_1^N(\P_\cdot(n_0<\tau_\d))\leq \varepsilon\right)\leq \frac{2(1+\sqrt{2})}{\varepsilon(c_0\varepsilon_0-\varepsilon)\sqrt{N}}.
\end{align*}
\end{lemma}

From this lemma (where we assume without loss of generality that $\delta\leq 1/2)$, from the Markov property applied to the particle system and since $\P_{\mu^N_0}(n_0<\tau_\d)> \varepsilon$ implies $\P_{\mu^N_0}(1<\tau_\d)> \varepsilon$, we deduce that, for any $\varepsilon\in(0,\delta\varepsilon_0)$,
\begin{align}
\P\left(\P_{\mu^N_{n+1}}(n_0<\tau_\d)\leq \varepsilon\right)&\leq \frac{2(1+\sqrt{2})}{\varepsilon(\varepsilon_0/2-\varepsilon)\sqrt{N}}\, \P\left(\P_{\mu^N_{n}}(n_0<\tau_\d)> \varepsilon\right)+(1-p_0)\P\left(\P_{\mu^N_{n}}(n_0<\tau_\d)\leq  \varepsilon\right)\nonumber\\
&\leq  \frac{2(1+\sqrt{2})}{\varepsilon(\varepsilon_0/2-\varepsilon)\sqrt{N}}+(1-p_0)\P\left(\P_{\mu^N_{n}}(n_0<\tau_\d)\leq  \varepsilon\right)\nonumber\\
&\leq \frac{2(1+\sqrt{2})}{p_0\varepsilon(\varepsilon_0/2-\varepsilon)\sqrt{N}}+(1-p_0)^{n+1}\P\left(\P_{\mu^N_{0}}(n_0<\tau_\d)\leq  \varepsilon\right),\label{eq:fromlemma1}
\end{align}
where the last line is obtained by iteration over $n$.

This and equation~\eqref{eq:last_eq} imply that, for any $\varepsilon\in\left(0,\varepsilon_0\delta/(c_3 e^{(n_1-n_0)\lambda_0})\right)$,
\begin{align*}
\E\left(\left|\mu^N_n(f)-\E_{\mu_{n-n_1}^N}(f(X_n)\mid n<\tau_\d)\right|\right)\leq 
&\frac{4(1+\sqrt{2})}{p_0c_3\varepsilon e^{(n_1-n_0)\lambda_0}(\varepsilon_0/2-c_3\varepsilon e^{(n_1-n_0)\lambda_0})\sqrt{N}}\\
&\quad +2(1+\sqrt{2})(1-p_0)^{n-n_1}\P\left(\P_{\mu^N_{0}}(n_0<\tau_\d)\leq  c_3\varepsilon e^{(n_1-n_0)\lambda_0}\right)\\
&\quad +\frac{2(1+\sqrt{2})}{\varepsilon\sqrt{N}}.
\end{align*}
Taking $\varepsilon=\varepsilon_0\delta/(c_3 e^{(n_1-n_0)\lambda_0})$ and assuming, without loss of generality, that $\delta\leq 1/4$, we obtain
\begin{align*}
\E\left(\left|\mu^N_n(f)-\E_{\mu_{n-n_1}^N}(f(X_n)\mid n<\tau_\d)\right|\right)\leq 
&\frac{16(1+\sqrt{2})}{p_0\varepsilon_0^2\delta\sqrt{N}}\\
&\quad +2(1+\sqrt{2})(1-p_0)^{n-n_1}\P\left(\P_{\mu^N_{0}}(n_0<\tau_\d)\leq \varepsilon_0\delta\right)\\
&\quad +\frac{2(1+\sqrt{2})c_2 e^{n_1\lambda_0}}{\varepsilon_0\delta\sqrt{N}}.
\end{align*}
Finally, using inequality~\eqref{eq:firststep} and taking 
\begin{align*}
n_1=\left\lfloor \frac{\ln N}{2\left(\lambda_0+\gamma\right)}\right\rfloor,
\end{align*}
straightforward computations implies the existence of a constant $C>0$ such that, for all $n\geq n_1$,
\begin{align*}
\E\left(\left|\mu^N_n(f)-\E_{\mu_0^N}(f(X_n)\mid n<\tau_\d)\right|\right)&\leq \frac{C N^\alpha}{\delta}+C\P\left(\P_{\mu^N_{0}}(n_0<\tau_\d)\leq \varepsilon_0\delta\right)
\end{align*}
with
\begin{align*}
\alpha={\frac{-\gamma}{2\left(\lambda_0+\gamma\right)}}<0.
\end{align*}

\medskip\noindent
Now, for $n\leq n_1$, we have
\begin{align*}
\E\left(\left|\mu^N_n(f)-\E_{\mu_0^N}(f(X_n)\mid n<\tau_\d)\right|\right)&\leq \E\left(2\wedge\frac{2(1+\sqrt{2})}{\P_{\mu^N_{0}}(n_1<\tau_\d)\sqrt{N}}\right)\\
&\leq \E\left(2\wedge \frac{2(1+\sqrt{2})e^{n_1\lambda_0}}{c_3\P_{\mu^N_{0}}(n_0<\tau_\d)\sqrt{N}}\right)\\
&\leq 2\P\left(\P_{\mu^N_{0}}(n_0<\tau_\d)\leq \varepsilon_0\delta\right)+\frac{2(1+\sqrt{2})e^{n_1\lambda_0}}{\varepsilon_0\delta\sqrt{N}}.
\end{align*}
Using the same computations as above, this concludes the proof of Theorem~\ref{thm:main-two}.
\end{proof}

\bigskip
\begin{proof}[Proof of Lemma~\ref{lem:controle_pos_part}]
Assume that $\P_{\mu^N_0}(1<\tau_\d)> \varepsilon$. We obtain from Theorem~\ref{thm:intro-main} that
\begin{align*}
\E\left(\left|\mu_1^N(\P_\cdot(n_0<\tau_\d))-\E_{\mu^N_0}\left(\P_{X_1}(n_0<\tau_\d)\mid 1<\tau_\d\right)\right|\right)&\leq \frac{2(1+\sqrt{2})\|\P_\cdot(n_0<\tau_\d)\|_\infty}{\varepsilon\sqrt{N}}.
\end{align*}
Markov's inequality thus implies that, for all $\delta>0$,
\begin{align*}
\P\left(\left|\mu_1^N(\P_\cdot(n_0<\tau_\d))-\E_{\mu^N_0}\left(\P_{X_1}(n_0<\tau_\d)\mid 1<\tau_\d\right)\right|\geq \delta\right)\leq \frac{2(1+\sqrt{2})}{\varepsilon\delta\sqrt{N}}
\end{align*}
and hence that
\begin{align*}
\P\left(\mu_1^N(\P_\cdot(n_0<\tau_\d))\leq \E_{\mu^N_0}\left(\P_{X_1}(n_0<\tau_\d)\mid 1<\tau_\d\right)- \delta\right)\leq \frac{2(1+\sqrt{2})}{\varepsilon\delta\sqrt{N}}.
\end{align*}
But, by Assumption~\eqref{eq:additional_assumption}, we have
$\P_{\mu^N_0}\left(\P_{X_1}(n_0<\tau_\d)\geq \varepsilon_0\mid 1<\tau_\d\right)\geq c_0$, so that $\E_{\mu^N_0}\left(\P_{X_1}(n_0<\tau_\d)\mid 1<\tau_\d\right)\geq c_0\varepsilon_0$. We deduce that
\begin{align*}
\P\left(\mu_1^N(\P_\cdot(n_0<\tau_\d))\leq \varepsilon_0c_0- \delta\right)\leq \frac{2(1+\sqrt{2})}{\varepsilon\delta\sqrt{N}}.
\end{align*}
Choosing $\delta=c_0\varepsilon_0-\varepsilon$, we finally obtained
\begin{align*}
\P\left(\mu_1^N(\P_\cdot(n_0<\tau_\d))\leq \varepsilon\right)\leq \frac{2(1+\sqrt{2})}{\varepsilon(c_0\varepsilon_0-\varepsilon)\sqrt{N}}.
\end{align*}

In the general case (when one does not have a good control on $\P_{\mu_0^N}(1<\tau_\d)$), the above strategy is bound to fail since we do not have a good control on the distance between the conditioned semi-group and the empirical distribution of the particle system. As a consequence, we need to take a closer look at Algorithm 1. As explained in the description of this algorithm, the position of the system at time $1$ is computed from the position of the system at time $0$ through several steps, each step being composed of two stages. We denote by $\bar{k}$ the number of steps needed to compute the position of the system at time $1$. 

 For any step $k\geq 1$, we denote by $X^{i,k}_0$ the position of the $i^{th}$ particle at the beginning of step $k$, by $b_i^{k}$ the state of the $i^{th}$ particle at the beginning of step $k$, and by $i_k\in \{i\in\{1,\ldots,N\},\,b^k_i=0\}$ the index of the particle chosen during the first stage of step $k$. With this notation, the process $\left((X^{i,k\wedge\bar{k}}_0,b_i^{k\wedge \bar{k}})_{i\in\{1,\ldots,N\}},i_{k\wedge\bar{k}}\right)_{k\in\N}$ is a Markov chain. In what follows, we denote by $(\cG_k)_{k\geq 1}$ the natural filtration of this Markov chain.

  We also introduce the quantities
\begin{align*}
N_k&=\sharp\{i\in\{1,\ldots,N\},\,b^k_i=1\text{ (at the beginning of the $k^{th}$ step)}\},\\
N'_k&=\sharp\{i\in\{1,\ldots,N\},\,b^k_i=1\text{ and }\P_{X^{i,k}_0}(n_0<\tau_\d)\geq \varepsilon_0\},\\
N''_k&=\sharp\{i\in\{1,\ldots,N\},\,b^k_i=1\text{ and }\P_{X^{i,k}_0}(n_0<\tau_\d)< \varepsilon_0\}.
\end{align*}
Of course, we have $N_k=N'_k+N''_k$ and, at the beginning of the first step, one has
$
N_1=N'_1=N''_1=0.
$
For any $k\geq 1$, conditionally to $\cG_k$ and on the event $k\leq \bar{k}$, the position of $X_0^{i_k,k+1}$  is chosen with respect to
\begin{align*}
\P_{X_0^{i_k,k}}\left(X_1\in\cdot\right)+\P_{X_0^{i_k,k}}(\tau_\d\leq 1)\frac{1}{N-1}\sum_{i=1,i\neq i_k}^N \delta_{X^{i,k}}.
\end{align*}
Hence, conditionally to $\cG_k$ and on the event $k\leq \bar{k}$, $(N'_{k+1},N''_{k+1})$ is equal to
\begin{align}
\label{eq:prob1}
\begin{cases}
(N'_{k},N''_{k})&\text{with prob. }\P_{X_0^{i_k,k}}(\tau_\d\leq 1)\frac{N-1-N_k}{N-1},\\
(N'_{k}+1,N''_{k})&\text{with prob. }\P_{X_0^{i_k,k}}\left(\P_{X_1}(n_0<\tau_\d)\geq \varepsilon_0\right)+\P_{X_0^{i_k,k}}(\tau_\d\leq 1)\frac{N'_k}{N-1},\\
(N'_k,N''_k+1)&\text{with prob. }\P_{X_0^{i_k,k}}\left(0<\P_{X_1}(n_0<\tau_\d)< \varepsilon_0\right)+\P_{X_0^{i_k,k}}(\tau_\d\leq 1)\frac{N''_k}{N-1}.
\end{cases}
\end{align}
From Assumption~\eqref{eq:additional_assumption}, we deduce that, conditionally to $\cG_k$ and on the event $k\leq \bar{k}$, $(N'_{k+1},N''_{k+1})$ is equal to
\begin{align}
\label{eq:prob2}
\begin{cases}
(N'_{k},N''_{k})&\text{with prob. }\P_{X_0^{i_k,k}}(\tau_\d\leq 1)\frac{N-1-N_k}{N-1},\\
(N'_{k}+1,N''_{k})&\text{with prob. }\geq c_0\P_{X_0^{i_k,k}}(1<\tau_\d)+\P_{X_0^{i_k,k}}(\tau_\d\leq 1)\frac{N'_k}{N-1},\\
(N'_{k},N''_{k}+1)&\text{with prob. }\leq (1-c_0)\P_{X_0^{i_k,k}}(1<\tau_\d)+\P_{X_0^{i_k,k}}(\tau_\d\leq 1)\frac{N''_k}{N-1}.
\end{cases}
\end{align}
Let us denote by $k_1,\ldots,k_n,\ldots$ the successive step numbers during which the sequence $(N_k)_{k\geq 1}$ jumps, that is
\begin{align*}
&k_1=\inf \{k\geq 1,\,N_{k+1}\neq N_k\},\\
&k_{n+1}=\inf \{k\geq k_n+1,\,N_{k+1}\neq N_k\},\ \forall n\geq 1.
\end{align*}
It is clear that, for all $n\geq 1$, $k_n+1$ is a stopping time with respect to the filtration~$\cG$.
We are interested in the sequence of random variables $(U_n)_{n\geq 1}$, defined by
\begin{align*} 
 U_n=N'_{k_n+1},\ \forall n\in \{1,\ldots,N\}.
\end{align*}
Conditionally to $\cG_{k_n+1}$ (the filtration $(\cG_k)_{k\geq 1}$ before the stopping time $k_n+1$), we deduce from~\eqref{eq:prob1} that, for all $x\in E$,
\begin{align*}
\P(U_{n+1}=U_n+1\mid\cG_{k_n+1},X_0^{i_{k_{n+1}},k_{n+1}}=x)&\geq \frac{\P_{x}(1<\tau_\d)c_0+\P_{x}(\tau_\d\leq 1)\frac{N'_{k_n+1}}{N-1}}{\P_{x}(1<\tau_\d)+\P_{x}(\tau_\d\leq 1)\frac{N_{k_n+1}}{N-1}}\\
&\geq c_0\wedge \frac{N'_{k_n+1}}{N_{k_n+1}}=c_0\wedge \frac{U_n}{n},
\end{align*}
since $N_{k_n+1}=n$ almost surely. We deduce that
\begin{align*}
\P(U_{n+1}=U_n+1\mid U_1,\ldots,U_n)&\geq c_0\wedge \frac{U_n}{n}\\
\P(U_{n+1}=U_n\mid U_1,\ldots,U_n)&=1-\P(U_{n+1}=U_n+1\mid U_1,\ldots,U_n)
                                  \leq 1-c_0\wedge\frac{U_n}{n}.
\end{align*}
Moreover, Assumption~\eqref{eq:additional_assumption} entails that $\P(U_1=1)\geq c_0> 0$.

As a consequence, there exists a coupling between $(U_n)_{n\in\N}$ and the Markov chain $(V_n)_{n\geq 1}$ with initial law $\P(V_1=1)=c_0$ and transition probabilities
\begin{align*}
\P(V_{n+1}=V_n+1\mid V_1,\ldots,V_n)&= c_0\wedge\frac{V_n}{n},\\
\P(V_{n+1}=V_n\mid V_1,\ldots,V_n)&=1-\P(V_{n+1}=V_n+1\mid V_1,\ldots,V_n)
\end{align*}
such that $U_n\geq V_n$, for all $n\in\{1,\ldots,N\}$. The process $(V_n/n)_{n\in\N}$ is a positive super-martingale (and a martingale if $c_0=1$) and hence it converges to a random variable $R_\infty$ almost surely as $n\rightarrow\infty$. Let us now prove that $R_\infty$ is not equal to zero almost surely.

Consider a P\'olya urn starting with $n_0$ balls with one white one, that is a Markov chain $(W_n)_{n\in\N}$ in $\N$ such that $W_0=1$ and
\begin{align*}
\P(W_{n+1}&=W_n+1\mid W_n)=\frac{W_n}{n+n_0},\ \forall n\geq 0,\\
\P(W_{n+1}&=W_n\mid W_n)=1-\frac{W_n}{n+n_0},\ \forall n\geq 0.
\end{align*}
It is well known that $(W_n/n)_{n\in\N}$ is a positive and bounded martingale which converges almost surely to a random variable $S_\infty$ distributed following a Beta distribution with parameters $(1,n_0-1)$. In particular, this implies that the event $\{W_k/k\leq c_0,\,\forall k\geq 0\}$ has a positive probability and that, conditionally to this event, $W_n/n$ converges to a positive random variable :
\begin{align*}
\mathbf{1}_{W_k/k\leq c_0,\,\forall k\geq 0}\frac{W_n}{n}\xrightarrow[n\rightarrow\infty]{a.s.} \mathbf{1}_{W_k/k\leq c_0,\,\forall k\geq 0}\,S_\infty.
\end{align*}
Since $(V_{n+n_0})_{n\in\N}$ and $(W_n)_{n\in\N}$ have the same transition probabilities at time $n$ from states $u\in\N$ such that $u/n\leq c_0$, there exists a coupling such that
\begin{align*}
\mathbf{1}_{V_0\geq 1}V_{n+n_0}\geq \mathbf{1}_{V_0\geq 1}\,\mathbf{1}_{W_k/k\leq c_0,\,\forall k\geq 0}\,W_n,\ \forall n\in\N\text{ almost surely,}
\end{align*}
and hence such that
\begin{align*}
\mathbf{1}_{V_0\geq 1}R_\infty\geq \mathbf{1}_{V_0\geq 1}\,\mathbf{1}_{W_k/k\leq c_0,\,\forall k\geq 0}\,S_\infty\text{ almost surely.}
\end{align*}
Since the right hand side is positive with positive probability, we deduce that $R_\infty$ is positive with positive probability. But $R_\infty>0$ implies that $\inf_{n\geq 0} V_n/n>0$, thus there exists $\delta>0$ such that 
\begin{align*}
\P\left(\inf_{n\geq 0}\frac{V_n}{n}> \delta\right)>0.
\end{align*}
Because of the relation between $V_n$ and $U_n$, we deduce that
\begin{align*}
p_0:=\P\left(\frac{N'_{k_N}}{N}> \delta\right)=\P\left(\frac{U_N}{N}> \delta\right)>0.
\end{align*}
By definition of $N'_{k_N}$, $\frac{N'_{k_N}}{N}> \delta$ implies $\mu_1^N(\P_\cdot(n_0<\tau_\d))> \delta\varepsilon_0$ and hence
\begin{align*}
\P\left(\mu_1^N(\P_\cdot(n_0<\tau_\d))\leq \delta\varepsilon_0\right)\leq 1-p_0.
\end{align*}
This concludes the proof of Lemma~\ref{lem:controle_pos_part}.

\end{proof}

\bigskip
\begin{proof}[Proof of Theorem~\ref{thm:invariant-quasi-stationary-distribution}]
We first prove the exponential ergodicity of the particle system and then deduce~\eqref{eq:dist_MN_QSD}.

Using Lemma~\ref{lem:controle_pos_part}, we know that, for any initial distribution of $(X^1_0,\ldots,X^N_0)$, 
\begin{align*}
\P(\P_{\mu_1^N}(n_0<\tau_\d)> \delta\varepsilon_0)\geq p_0.
\end{align*}
On the event $\P_{\mu_1^N}(n_0<\tau_\d)> \delta\varepsilon_0$, there exists at least one particle satisfying $\P_{X^i_1}(n_0<\tau_\d)>\delta\varepsilon_0$. Let us denote by $I\subset\{1,\ldots,N\}$ the set of  indexes of such particles and by $J=\{1,\ldots,N\}\setminus I$ the set of indexes $j$ such that  $\P_{X^j_1}(n_0<\tau_\d)\leq \delta\varepsilon_0$.

The probability that the $\sharp J$ first steps of Algorithm~1 concern the indexes of~$\sharp J$ in strictly increasing order is strictly lowered by $1/N^{\sharp J}$ and hence by $1/N^N$. For each of this step, the probability that the chosen particle with index in $J$ is killed and then is sent to  the position of a particle with index in $I$ is lowered by $(1-\delta\varepsilon_0)\sharp I/N$ and hence by $(1-\delta\varepsilon_0)/N$. Overall, the probability that, after the $\sharp J$ first steps of Algorithm~1, all the particles with index in $J$ have jumped on a particle with index in $I$ is bounded below by $1/N^N\times (1-\delta\varepsilon_0)^{\sharp J}/N^{\sharp J}$ and hence by $(1-\delta\varepsilon_0)^N/N^{2N}$.

On this event, the probability that, for each next step in the algorithm up to time $n_0+1$, the chosen particle jumps without being absorbed is bounded below by $(\delta\varepsilon_0)^N$. But, using~\cite[Theorem~2.1]{CV2016PTRF}, we know that, under Assumption~\ref{eq:expo-cv}, there exist a probability measure $\nu$ on $E$ and a constant $c_1>0$ such that $\P_x(X_{n_0}\in\cdot\mid n_0<\tau_\d)\geq c_1\nu(\cdot)$. Since the particles are independent on the event where none of them is killed, we finally deduce that the distribution of the particle system at time $n_0+1$ satisfies
\begin{align*}
\P\left((X^1_{n_0+1},\ldots,X^1_{n_0+1})\in \cdot\right)\geq p_0(\delta\varepsilon_0^N(1-\delta\varepsilon_0)^N/N^{2N}c_1^N\nu^{\otimes N}(\cdot).
\end{align*}
Classical coupling criteria (see for instance~\cite{Meyn2009}) entails the exponential ergodicity of the particle system.

\medskip
Let us now prove that~\eqref{eq:dist_MN_QSD} holds. Consider $x\in E$ and $\delta >0$ such that $\P_x(n_0<\tau_\d)>\varepsilon_0\delta$. Then, applying Theorem~\ref{thm:main-two} to the particle system with initial position $(X^1_0,\ldots,X^N_0)=(x,\ldots,x)$, we deduce that, for all $n\geq 0$,
\begin{align*}
\E\left(\left|\mu^N_n(f)-\E_{x}(f(X_n)\mid n<\tau_\d)\right|\right)&\leq \frac{C N^\alpha}{\delta}.
\end{align*}
Using the exponential ergodicity of the particle system and the exponential convergence~\eqref{eq:expo-cv}, we deduce that
\begin{align*}
\E_{M^N}\left(\left|\mu^N_0(f)-\nu_{QSD}(f)\right|\right)\leq \frac{CN^\alpha}{\delta}+2C_Ne^{-\gamma_N n}+Ce^{-\gamma n}.
\end{align*}
Letting $n$ tend toward infinity implies~\eqref{eq:dist_MN_QSD} and concludes the proof of Theorem~\ref{thm:invariant-quasi-stationary-distribution}. 
\end{proof}

\section{Example: Absorbed neutron transport process}

\label{sec:example_neutron}

The propagation of neutrons in fissible media is typically modeled by neutron transport systems, where the trajectory of the particle
is composed of straight exponential paths between random changes of directions~\cite{DautrayLions1993,Zoia2011}.  The behavior of a neutron before its absorption by a medium is related to the behavior of neutron tranport before extinction,
where extinction corresponds to the exit of a neutron from a bounded set $D$.

We recall the setting of the neutron transport process studied in~\cite{CV2016PTRF}. Let $D$ be an open connected bounded domain of
$\mathbb{R}^2$, let $S^2$ be the unit sphere of $\mathbb{R}^2$ and $\sigma(du)$ be the uniform probability measure on $S^2$. We
consider the Markov process $(X_t,V_t)_{t\geq 0}$ in $D\times S^2$ constructed as follows: $X_t=\int_0^t V_s\,ds$ and the velocity
$V_t\in S^2$ is a pure jump Markov process, with constant jump rate $\lambda>0$ and uniform jump probability distribution $\sigma$.
In other words, $V_t$ jumps to i.i.d. uniform values in $S^2$ at the jump times of a Poisson process. At the first time where
$X_t\not\in D$, the process immediately jumps to the cemetery point $\d$, meaning that the process is absorbed at the boundary of
$D$. An example of path of the process $(X,V)$ is shown in Fig.~\ref{fig:sample-path}. For all $x\in D$ and $u\in S^2$, we denote by
$\mathbb{P}_{x,u}$ (resp.\ $\mathbb{E}_{x,u}$) the distribution of $(X,V)$ conditionned on $(X_0,V_0)=(x,u)$ (resp.\ the expectation
with respect to $\mathbb{P}_{x,u}$).

\begin{figure}[h]
  \begin{center}
    \vspace{0.5cm}
    \includegraphics[height=6.5cm]{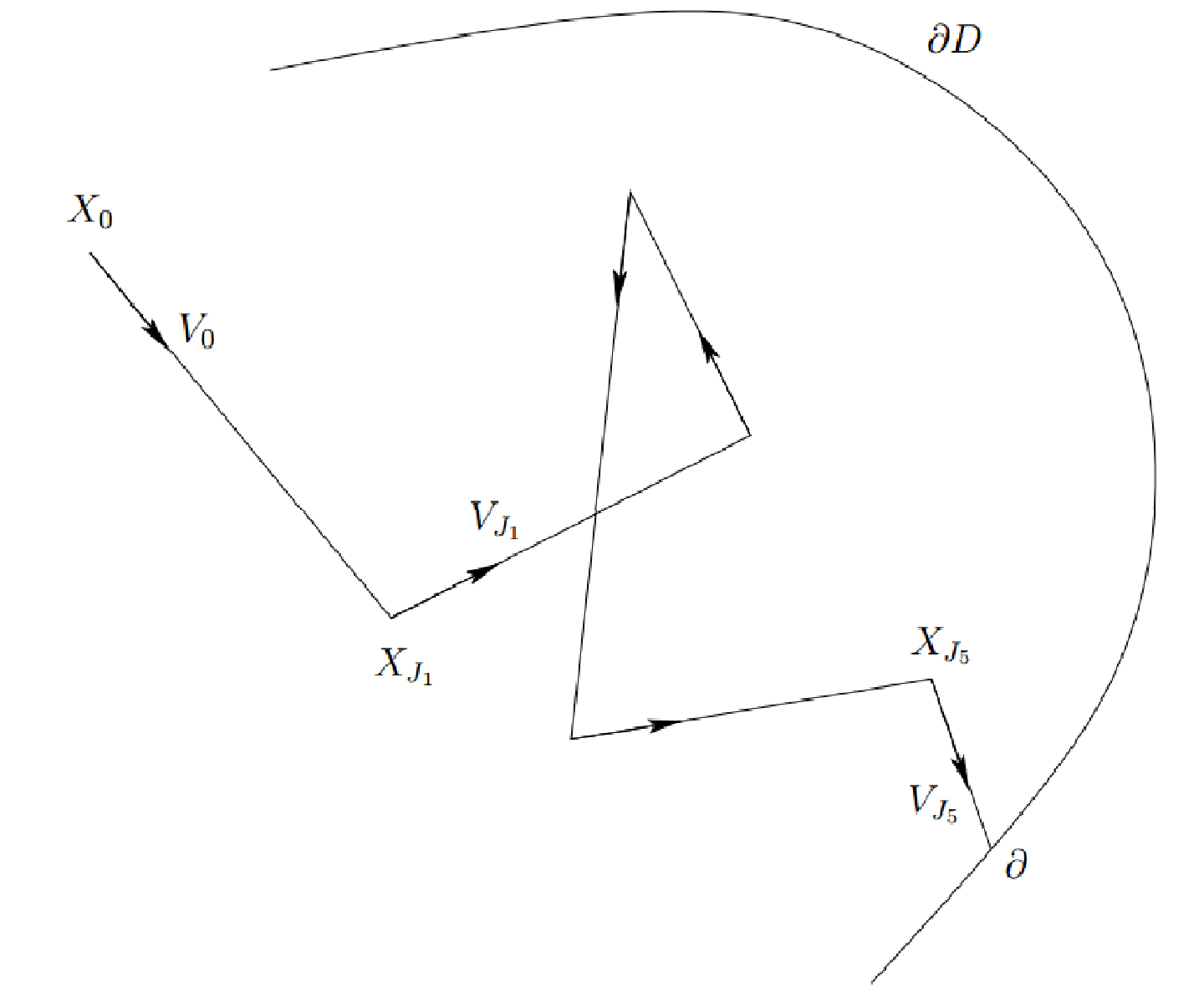}
  \end{center}
  \caption{A sample path of the neutron transport process $(X,V)$. The times $J_1<J_2<\ldots$ are the successive jump times of $V$.}
  \label{fig:sample-path}
\end{figure}

We also assume the following condition on the boundary of the bounded open set $D$. This is an interior cone type condition satisfied for example by convex open sets of $\R^2$ and by open sets with $C^2$ boundaries.

\paragraph{Assumption (H) on $D$.}
We assume that there exists $\varepsilon>0$ such that
\begin{itemize}
\item $D_{\varepsilon}:=\{x\in D:d(x,\partial D)>\varepsilon\}$ is non-empty and connected;
\item there exists $0<s_\varepsilon<t_\varepsilon$ and $\underline{\sigma}>0$ such that, for all $x\in D\setminus
  D_\varepsilon$, there exists $K_x\subset S^2$ measurable such that $\sigma(K_x)\geq \underline{\sigma}$ and for all $u\in K_x$,
  $x+su\in D_\varepsilon$ for all $s\in[s_\varepsilon,t_\varepsilon]$ and $x+su\not\in \partial D$ for all $s\in[0,s_\varepsilon]$.
\end{itemize}

\bigskip

The next proposition implies that the approximation method introduced in this paper converges uniformly in time toward the conditional distribution of $(X,V)$. 

\begin{proposition}
\label{prop:neutron}
The absorbed Markov process $(X,V)$ satisfies the assumptions of Theorem~\ref{thm:main-two}.
\end{proposition}

\begin{proof}[Proof of Proposition~\ref{prop:neutron}]
Theorem~4.3 from~\cite{CV2016PTRF} states that the Markov process $(X,V)$ satisfies the exponential convergence condition~\eqref{eq:expo-cv}. It only remains to check that that Assumption~\eqref{eq:additional_assumption} is also satisfied.

 By~\cite[(4.3)]{CV2016PTRF}, for any $n_0\geq 1$, there exists a constant $c_{n_0}>0$ such that
\begin{align}
\label{eq:def-c_n0}
c_{n_0}:=\inf_{(x,u)\in D_\varepsilon\times S^2} \P_{(x,u)}(1+n_0< \tau_\d)>0.
\end{align}
In particular, we deduce that, for all $n_0\geq 0$,
\begin{align}
\label{eq:add-assumption-neutron1}
\P_{(x,u)}(\P_{(X_1,V_1)}(n_0<\tau_\d)\mid 1<\tau_\d)\geq c_{n_0},\quad\forall (x,u)\in D_{\varepsilon}\times S^2.
\end{align}
Now, fix $(x,u)\notin D_\varepsilon\times S^2$ and consider the first (deterministic) time $s(x,u)$ when the ray starting from $x$ with direction $u$ hits $D_\varepsilon$ or $\d D$, defined by
\begin{align*}
s(x,u)=\inf\{t\geq 0\,,\text{ such that }x+tu\in D_\varepsilon\cup \d D\}.
\end{align*}
Let us first assume that $x+s(x,u)u\in\d D_\varepsilon$. Then $x+(s(x,u)+\eta)u\in D_\varepsilon$ for some $\eta>0$ and hence
\begin{align*}
\P_{(x,u)}(X_{s(x,u)+\eta}\in D_\varepsilon)&\geq \P\left(J_1>s(x,u)+\eta\right)= e^{-(s(x,u)+\eta)}\geq e^{-\text{diam}(D)},
\end{align*}
where $J_1<J_2<\cdots$ denotes the successive jump times of the process $V$. Using the Markov property, we deduce from~\eqref{eq:def-c_n0} that
\begin{align}
\label{eq:not-conditioned1}
\P_{(x,u)}(1+n_0<\tau_\d)\geq \P(s+\eta+1+n_0<\tau_\d)\geq e^{-\text{diam}(D)}\,c_{n_0}.
\end{align}
This implies that, for all $n_0\geq 1$ and all $(x,u)\notin D_{\varepsilon}\times S^2$ such that $x+s(x,u)u\in\d D_\varepsilon$,
\begin{align}
\label{eq:add-assumption-neutron2}
\P_{(x,u)}(\P_{(X_1,V_1)}(n_0<\tau_\d)\mid 1<\tau_\d)\geq e^{-\text{diam}(D)} c_{n_0}.
\end{align}
Let us now assume that $x+s(x,u)u\in\d D$. Using Assumption~(H), we obtain
\begin{align*}
\P_{(x,u)}(\exists t\in [0,s)\text{ s.t. } X_t+s(X_t,V_t)V_t\in \d D_\varepsilon)&
\geq \P_{(x,u)} (J_1<s,\,V_{J_1}\in K_{x+uJ_1},\,J_2>s)\\
&\geq (1-e^{-s})\underline{\sigma}e^{-s}.
\end{align*}
Using the strong Markov property and~\eqref{eq:not-conditioned1}, we deduce that
\begin{align*}
\P_{(x,u)}(1+n_0<\tau_\d)&\geq (1-e^{-s})\underline{\sigma}e^{-s}e^{-\text{diam}(D)}c_{n_0}\geq (1-e^{-s\wedge 1})\underline{\sigma}e^{-2\text{diam}(D)}c_{n_0}.
\end{align*}
If $s\leq 1$, then $1<\tau_\d$ implies that the process jumps at least one time before reaching $\d D$, that is before time $s$, so that $\P_{(x,u)}(1<\tau_\d)\leq \P(J_1\leq s)=1-e^{-s}$. Using the Markov property, we deduce that, for $s\leq 1$ or $s>1$, 
\begin{align*}
\P_{(x,u)}(\P_{(X_1,V_1)}(n_0<\tau_\d)\mid 1<\tau_\d)&=
\P_{(x,u)}(1+n_0<\tau_\d\mid 1<\tau_\d)\\
&\geq (1-e^{-1})\underline{\sigma}e^{-2\text{diam}(D)}c_{n_0}.
\end{align*}
This, Equations~\eqref{eq:add-assumption-neutron1} and~\eqref{eq:add-assumption-neutron2} together imply that Assumption~\eqref{eq:additional_assumption} is fulfilled for all $(x,u)\in D,S^2$. This concludes the prood of Proposition~\ref{prop:neutron}.
\end{proof}


\begin{thebibliography}{10}

\bibitem{Asselah2016}
A.~Asselah, P.~A. Ferrari, P.~Groisman, and M.~Jonckheere.
\newblock Fleming-{V}iot selects the minimal quasi-stationary distribution: The
  {G}alton-{W}atson case.
\newblock {\em Ann. Inst. H. Poincar\'e Probab. Statist.}, 52(2):647--668, 05
  2016.

\bibitem{BieniekBurdzyPal2012}
M.~Bieniek, K.~Burdzy, and S.~Pal.
\newblock Extinction of fleming-viot-type particle systems with strong drift.
\newblock {\em Electron. J. Probab.}, 17:no. 11, 1--15, 2012.

\bibitem{Burdzy1996}
K.~Burdzy, R.~Holyst, D.~Ingerman, and P.~March.
\newblock Configurational transition in a fleming-viot-type model and
  probabilistic interpretation of laplacian eigenfunctions.
\newblock {\em J. Phys. A}, 29(29):2633--2642, 1996.

\bibitem{Burdzy2000}
K.~Burdzy, R.~Ho{\l}yst, and P.~March.
\newblock A {F}leming-{V}iot particle representation of the {D}irichlet
  {L}aplacian.
\newblock {\em Comm. Math. Phys.}, 214(3):679--703, 2000.

\bibitem{CCV2016}
N.~{Champagnat}, A.~{Coulibaly-Pasquier}, and D.~{Villemonais}.
\newblock {Exponential convergence to quasi-stationary distribution for
  multi-dimensional diffusion processes}.
\newblock {\em ArXiv e-prints}, Mar. 2016.

\bibitem{CV2016c}
N.~{Champagnat} and D.~{Villemonais}.
\newblock {Exponential convergence to quasi-stationary distribution for
  absorbed one-dimensional diffusions with killing}.
\newblock {\em ArXiv e-prints}, Oct. 2015.

\bibitem{CV2015PNMmd}
N.~{Champagnat} and D.~{Villemonais}.
\newblock {Quasi-stationary distribution for multi-dimensional birth and death
  processes conditioned to survival of all coordinates}.
\newblock {\em ArXiv e-prints}, Aug. 2015.

\bibitem{CV2016b}
N.~{Champagnat} and D.~{Villemonais}.
\newblock {Uniform convergence of conditional distributions for absorbed
  one-dimensional diffusions}.
\newblock {\em ArXiv e-prints}, June 2015.

\bibitem{CV2016PTRF}
N.~Champagnat and D.~Villemonais.
\newblock Exponential convergence to quasi-stationary distribution and
  {$Q$}-process.
\newblock {\em Probab. Theory Related Fields}, 164(1-2):243--283, 2016.

\bibitem{CV2016a}
N.~{Champagnat} and D.~{Villemonais}.
\newblock {Uniform convergence of penalized time-inhomogeneous Markov
  processes}.
\newblock {\em ArXiv e-prints}, Mar. 2016.


\bibitem{ColletMartinezSanMartin2013}
P.~Collet, S.~Mart{\'{\i}}nez, and J.~San~Mart{\'{\i}}n.
\newblock {\em Quasi-stationary distributions}.
\newblock Probability and its Applications (New York). Springer, Heidelberg,
  2013.
\newblock Markov chains, diffusions and dynamical systems.

\bibitem{DautrayLions1993}
R.~Dautray and J.-L. Lions.
\newblock {\em Mathematical analysis and numerical methods for science and
  technology. {V}ol. 6}.
\newblock Springer-Verlag, Berlin, 1993.

\bibitem{DelMoralWP}
P.~Del~Moral.
\newblock Feynman-{K}ac models and interacting particle systems.
\newblock \url{http://web.maths.unsw.edu.au/~peterdel-moral/simulinks.html}.


\bibitem{DelMoral1998}
P.~Del~Moral.
\newblock Measure-valued processes and interacting particle systems. Application to nonlinear filtering problems.
\newblock {\em Ann. Appl. Probab.}, 8(2):438--495, 1998.



\bibitem{DelMoral2003}
P.~Del~Moral and L.~Miclo.
\newblock Particle approximations of {L}yapunov exponents connected to
  {S}chr\"odinger operators and {F}eynman-{K}ac semigroups.
\newblock {\em ESAIM Probab. Stat.}, 7:171--208, 2003.


\bibitem{Ferrari2007}
P.~A. Ferrari and N.~Mari{\'c}.
\newblock Quasi stationary distributions and {F}leming-{V}iot processes in
  countable spaces.
\newblock {\em Electron. J. Probab.}, 12:no. 24, 684--702 (electronic), 2007.

\bibitem{Grigorescu2004}
I.~Grigorescu and M.~Kang.
\newblock Hydrodynamic limit for a {F}leming-{V}iot type system.
\newblock {\em Stochastic Process. Appl.}, 110(1):111--143, 2004.

\bibitem{meleard-villemonais-12}
S.~M{\'e}l{\'e}ard and D.~Villemonais.
\newblock Quasi-stationary distributions and population processes.
\newblock {\em Probab. Surv.}, 9:340--410, 2012.

\bibitem{Meyn2009}
S.~Meyn and R.~Tweedie.
\newblock {\em Markov chains and stochastic stability}.
\newblock Cambridge University Press New York, NY, USA, 2009.

\bibitem{Rousset2006}
M.~Rousset.
\newblock On the control of an interacting particle estimation of
  {S}chr\"odinger ground states.
\newblock {\em SIAM J. Math. Anal.}, 38(3):824--844 (electronic), 2006.

\bibitem{vanDoorn2013}
E.~A. van Doorn and P.~K. Pollett.
\newblock Quasi-stationary distributions for discrete-state models.
\newblock {\em European J. Oper. Res.}, 230(1):1--14, 2013.

\bibitem{Villemonais2011}
D.~Villemonais.
\newblock Interacting Particle Systems and Yaglom Limit Approximation of Diffusions with Unbounded Drift.
\newblock {\em Electron. J. Probab.}, 16:no. 61, 1663--1692, 2011.

\bibitem{Villemonais2014}
D.~Villemonais.
\newblock General approximation method for the distribution of markov processes
  conditioned not to be killed.
\newblock {\em ESAIM: Probability and Statistics}, eFirst, 2 2014.

\bibitem{Villemonais2015}
D.~Villemonais.
\newblock Minimal quasi-stationary distribution approximation for a birth and
  death process.
\newblock {\em Electron. J. Probab.}, 20:no. 30, 1--18, 2015.

\bibitem{Zoia2011}
A.~Zoia, E.~Dumonteil, and A.~Mazzolo.
\newblock Collision densities and mean residence times for {$d$}-dimensional
  exponential flights.
\newblock {\em Phys. Rev. E}, 83:041137, Apr 2011.

\end{thebibliography}
\end{document}